\documentclass[12pt, reqno]{amsart}

\usepackage{amsmath, amsthm, amscd, amsfonts, amssymb, graphicx, color}
\usepackage[bookmarksnumbered, colorlinks, plainpages]{hyperref}% this package creats hyperlinks.
\usepackage{mathtools}
\usepackage{enumerate}% for creating numbered list.
\usepackage{amscd} % This code allows you to draw the commutative diagrams
\usepackage{xy}
\xyoption{all}
\newtheorem{theorem}{Theorem}[section]
\newtheorem{lemma}[theorem]{Lemma}

\newtheorem{corollary}[theorem]{Corollary}
\theoremstyle{definition}

\newtheorem{remark}[theorem]{Remark}
\begin{document}
%-----------------------------------------------PAPER / ARTICLE FORMATINGS---------------------------------------------------
\setcounter{page}{1}

\title[ ]{Complete monotonicity properties of a function involving the polygamma function}

\author[K. Nantomah]{Kwara Nantomah}

\address{ Department of Mathematics, Faculty of Mathematical Sciences, University for Development Studies, Navrongo Campus, P. O. Box 24, Navrongo, UE/R, Ghana. }
\email{\textcolor[rgb]{0.00,0.00,0.84}{knantomah@uds.edu.gh}}

%\address{$^{2}$ Department of Mathematics, University of Mitrovica "Isa Boletini", Kosovo. }
%\email{\textcolor[rgb]{0.00,0.00,0.84}{fmerovci@yahoo.com}}

%\address{$^{3}$ Department of Statistics, Faculty of Mathematical Sciences, University for Development Studies, Navrongo Campus, P. O. Box 24, Navrongo, UE/R, Ghana. }
%\email{\textcolor[rgb]{0.00,0.00,0.84}{sulemanstat@gmail.com}}

%\dedicatory{This paper is dedicated to Professor ABCD}



\date{Received: xxxxxx; Revised: xxxxxx; Accepted: xxxxxx.
%\newline \indent $^{*}$ Corresponding author}
}

\begin{abstract}
In this paper, we study completete monotonicity properties of the function $f_{a,k}(x)=\psi^{(k)}(x+a) - \psi^{(k)}(x) - \frac{ak!}{x^{k+1}}$, where $a\in(0,1)$ and $k\in \mathbb{N}_0$. Specifically, we consider the cases for $k\in \{ 2n: n\in \mathbb{N}_0 \}$ and  $k\in \{ 2n+1: n\in \mathbb{N}_0 \}$. Subsequently, we deduce some  inequalities involving the polygamma functions.  
\end{abstract} \maketitle

%--------------------------------------------------------BODY OF ARTICLE --------------------------------------------------------

%%%%%%%%%%%%%%%%%%%%%%%%%%%%%%%%%%%%%%%%%%%%%%%%%%%%
\section{Introduction and Preliminaries}
%%%%%%%%%%%%%%%%%%%%%%%%%%%%%%%%%%%%%%%%%%%%%%%%%%%%

%\noindent
%Throughout this paper, we use the notations $\mathbb{N}=\{1,2,3,4,\dots \}$, $\mathbb{N}_0=\mathbb{N}\cup \{0\}$ and $\mathbb{R}=(-\infty, \infty)$.

\noindent
The classical Gamma function,  which is an extension of the factorial notation to noninteger values is usually defined as 
\begin{align*}
\Gamma(x)&=\int_{0}^{\infty} t^{x-1}e^{-t}\,dt,  \quad x>0,
\end{align*}
and satisfying the basic property
\begin{align*}
\Gamma(x+1)&=x\Gamma(x), \quad x>0.  
\end{align*}
Its logarithmic derivative, which is called the Psi or digamma function is defined as  (see \cite[p. 258-259]{Abramowitz-Stegun-1972} and \cite[p. 139-140]{Olver-etal-2010-CUP})
\begin{align}
\psi(x)=\frac{d}{dx}\ln \Gamma(x)
&= -\gamma + \int_{0}^{\infty} \frac{e^{-t} - e^{-xt}}{1-e^{-t}}\,dt, \quad x>0,   \label{eqn:Digamma-Integral-Rep} \\
&= -\gamma - \frac{1}{x} + \sum_{k=1}^{\infty}\frac{x}{k(k+x)}, \quad x>0,  \nonumber %\label{eqn:Digamma-Series-Rep-2}
\end{align}
where $\gamma=\lim_{n \rightarrow \infty} \left( \sum_{k=1}^{n}\frac{1}{k}- \ln n \right)=0.577215664...$ is the Euler-Mascheroni's constant. Derivatives of the Psi function, which are called polygamma functions are given as  \cite[p. 260]{Abramowitz-Stegun-1972}
\begin{align}
\psi^{(n)}(x)
&= (-1)^{n+1}\int_{0}^{\infty} \frac{t^{n}e^{-xt}}{1-e^{-t}}\,dt, \quad x>0,  \label{eqn:Polygamma-Integral-Rep} \\
& = (-1)^{n+1}n! \sum_{k=0}^{\infty}\frac{1}{(k+x)^{n+1}}, \quad x>0,  \nonumber %\label{eqn:Polygamma-Series-Rep}
\end{align}
satisfying the functional equation \cite[p. 260]{Abramowitz-Stegun-1972} 
\begin{equation}\label{eqn:funct-eqn-polygamma}
\psi^{(n)}(x+1) = \psi^{(n)}(x) + \frac{(-1)^nn!}{x^{n+1}}, \quad x>0,
\end{equation}
where $n\in \mathbb{N}_0$ and $\psi^{(0)}(x)\equiv \psi(x)$. Here, and for the rest of this paper, we use the notations: $\mathbb{N}=\{1,2,3,4,\dots \}$, $\mathbb{N}_0=\mathbb{N}\cup \{0\}$ and $\mathbb{R}=(-\infty, \infty)$.

\noindent
Also, it is well known in the literature that the integral
\begin{equation}\label{eqn:integral-rep-vip}
\frac{n!}{x^{n+1}} = \int_{0}^{\infty}t^{n}e^{-xt}\,dt,
\end{equation}
holds for $x>0$ and $n\in \mathbb{N}_0$. See for instance \cite[p. 255]{Abramowitz-Stegun-1972}. 

\noindent
In \cite{Qiu-Vuorinen-2004-MC}, Qiu and Vuorinen established among other things that the function
\begin{equation}\label{eqn:Qiu-Vuorinen-Ineq}
h_{1}=\psi \left( x+\frac{1}{2}\right) - \psi \left( x\right) - \frac{1}{2x}, 
\end{equation}
is strictly decreasing and convex on $(0,\infty)$. Motivated by this result, Mortici \cite{Mortici-2010-AUA} proved a more generalized and deeper result which states that, the function
\begin{equation}\label{eqn:Mortici-Gen-Ineq}
f_{a}=\psi(x+a) - \psi(x) - \frac{a}{x}, \quad a\in(0,1),
\end{equation}
is strictly completely monotonic on $(0,\infty)$. In this paper, the objective is to extend Mortici's results to the polygamma functions. Particularly, we study completete monotonicity properties of the function $f_{a,k}(x)=\psi^{(k)}(x+a) - \psi^{(k)}(x) - \frac{ak!}{x^{k+1}}$, where $a\in(0,1)$ and $k\in \mathbb{N}_0$, by considering the cases for $k\in \{ 2n: n\in \mathbb{N}_0 \}$ and $k\in \{ 2n+1: n\in \mathbb{N}_0 \}$.  Unlike Mortici's work, the techniques of the present work are simple and do not rely on the Hausdorff-Bernstein-Widder theorem.

%%%%%%%%%%%%%%%%%%%%%%%%%%%%%%%%%%%%%%%%%%%%%%%%%%%%
\section{Main Results}
%%%%%%%%%%%%%%%%%%%%%%%%%%%%%%%%%%%%%%%%%%%%%%%%%%%%

\noindent
We present our findings in this section by starting with the following lemma.

%------- Lemma:  ---------------------
\begin{lemma}\label{lem:Ratio-of-exp-funct}
Let a function $q_{\alpha, \beta}(t)$ be defined as
\begin{equation}\label{eqn:Ratio-of-exp-funct}
q_{\alpha, \beta}(t)=
\begin{cases}% Function in piecewise form are written using this code. Use "\\" to break entries
\frac{e^{-\alpha t}-e^{-\beta t}}{1-e^{-t}}, & \quad t\neq 0, \\
\beta-\alpha, & \quad t=0 ,
\end{cases}
\end{equation}
where $\alpha$, $\beta$ are real numbers such that $\alpha \neq \beta$ and $(\alpha, \beta)\notin \{ (0,1), (1,0) \}$. Then $q_{\alpha, \beta}(t)$ is increasing on $(0,\infty)$ if and only if $(\beta-\alpha)(1-\alpha-\beta)\geq0$ and $(\beta-\alpha)(|\alpha-\beta|-\alpha-\beta)\geq0$.
\end{lemma}

\begin{proof}
See \cite[Theorem 1.16]{Qi-2010-JIA} or \cite[Proposition 4.1]{Qi-Luo-2012-BJMA}.
\end{proof}

%------- Lemma:  ---------------------
\begin{lemma}\label{lem:Ratio-of-exp-funct-particular}
Let $a\in(0,1)$. Then the inequality
\begin{equation}\label{eqn:Ratio-of-exp-funct-particular}
a< \frac{1-e^{-at}}{1-e^{-t}} < 1 ,
\end{equation}
holds for $t\in(0,\infty)$.
\end{lemma}

\begin{proof}
Note that the function $h(t)=\frac{1-e^{-at}}{1-e^{-t}}$ which is obtained from Lemma \ref{lem:Ratio-of-exp-funct} by letting $\alpha=0$ and $\beta=a\in(0,1)$ is increasing on $(0,\infty)$. Also,
\begin{equation*}
\lim_{t\rightarrow 0^+}h(t)=a \quad \text{and} \quad \lim_{t\rightarrow \infty}h(t)=1 .
\end{equation*}
Then for $t\in(0,\infty)$, we have 
\begin{equation*}
a=\lim_{t\rightarrow 0^+}h(t)=h(0)<h(t)< h(\infty)=\lim_{t\rightarrow \infty}h(t)=1 ,
\end{equation*}
which gives \eqref{eqn:Ratio-of-exp-funct-particular}.
\end{proof}

%-------------------------- Theorem:--------------------------------------
\begin{theorem}\label{thm:Gen-Mortici-CM}
Let $a\in(0,1)$ and $k\in \{ 2n: n\in \mathbb{N}_0 \}$. Then the function
\begin{equation}\label{eqn:Gen-Mortici-CM}
f_{a,k}(x)=\psi^{(k)}(x+a) - \psi^{(k)}(x) - \frac{ak!}{x^{k+1}},
\end{equation}
is strictly completely monotonic on $(0,\infty)$.
\end{theorem}

\begin{proof}
Recall that a function $f:(0,\infty)\rightarrow \mathbb{R}$ is said to be completely monotonic on $(0,\infty)$  if $f$ has derivatives of all order and $(-1)^{n}f^{(n)}(x)\geq0$  for all $x\in(0,\infty)$ and $n\in \mathbb{N}_0$. Let $a\in(0,1)$ and $k\in \{ 2n: n\in \mathbb{N}_0 \}$. Then, by repeated differentiation and by using \eqref{eqn:Polygamma-Integral-Rep} and \eqref{eqn:integral-rep-vip}, we obtain
\begin{align*}
f_{a,k}^{(n)}(x)&=\psi^{(k+n)}(x+a) - \psi^{(k+n)}(x) - \frac{(-1)^{n}a(k+n)!}{x^{k+n+1}}  \\
&= (-1)^{k+n+1}\int_{0}^{\infty}\frac{t^{k+n}e^{-(x+a)t}}{1-e^{-t}}\,dt -  (-1)^{k+n+1}\int_{0}^{\infty}\frac{t^{k+n}e^{-xt}}{1-e^{-t}}\,dt   \\
& \quad -  (-1)^na\int_{0}^{\infty}t^{k+n}e^{-xt}\,dt.  
\end{align*}
Then
\begin{align*}
(-1)^nf_{a,k}^{(n)}(x)&= -\int_{0}^{\infty}\frac{t^{k+n}e^{-xt}e^{-at}}{1-e^{-t}}\,dt + \int_{0}^{\infty}\frac{t^{k+n}e^{-xt}}{1-e^{-t}}\,dt - a\int_{0}^{\infty}t^{k+n}e^{-xt}\,dt  \\
&= \int_{0}^{\infty} \left[\frac{1-e^{-at}}{1-e^{-t}} - a \right]t^{k+n}e^{-xt}\,dt \\
&>0,
\end{align*}
as a result of Lemma \ref{lem:Ratio-of-exp-funct-particular}. Alternatively, we could proceed as follows.
\begin{align*}
(-1)^nf_{a,k}^{(n)}(x)&= \int_{0}^{\infty} \left[\frac{1-e^{-at}}{1-e^{-t}} - a \right]t^{k+n}e^{-xt}\,dt   \\
&= a\int_{0}^{\infty} \left[\frac{1-e^{-at}}{at} - \frac{1-e^{-t}}{t} \right]\frac{t^{k+n+1}e^{-xt}}{1-e^{-t}}\,dt \\
&>0.
\end{align*}
Notice that, since the function $\frac{1-e^{-t}}{t}$ is strictly decreasing on $(0,\infty)$, then for $a\in(0,1)$, we have $\frac{1-e^{-at}}{at} > \frac{1-e^{-t}}{t}$.
\end{proof}

%-------------------Remark -------------------------
\begin{remark}
Since every completely monotonic function is convex and decreasing, it follows that $f_{a,k}(x)$ is strictly convex and strictly decreasing on $(0,\infty)$. 
\end{remark}

%-------------------------- Corollary:--------------------------------------
\begin{corollary}\label{cor:Gen-Vuorinen-type-ineq}
The inequality
\begin{equation}\label{eqn:Gen-Vuorinen-type-ineq}
\frac{ak!}{x^{k+1}} < \psi^{(k)}(x+a) - \psi^{(k)}(x) < 
\psi^{(k)}(a) - \psi^{(k)}(1) + k!\left(\frac{a}{x^{k+1}} + \frac{1}{a^{k+1}} - a\right),
\end{equation}
holds for $a\in(0,1)$, $k\in \{ 2n: n\in \mathbb{N}_0 \}$ and $x\in(1,\infty)$.
\end{corollary}

\begin{proof}
Since $f_{a,k}(x)$  is decreasing, then for $x\in[1,\infty)$ and by \eqref{eqn:funct-eqn-polygamma}, we obtain
\begin{align*}
0=\lim_{x\rightarrow \infty}f_{a,k}(x) < f_{a,k}(x) < f_{a,k}(1)&=\psi^{(k)}(a+1) - \psi^{(k)}(1) - ak! \\
&=\psi^{(k)}(a) - \psi^{(k)}(1) + \frac{k!}{a^{k+1}} - ak! ,
\end{align*}
which completes the proof.
\end{proof}

%-------------------Remark -------------------------
\begin{remark}
If $a=\frac{1}{2}$ and $k=0$ in Corollary \ref{cor:Gen-Vuorinen-type-ineq}, then we obtain
\begin{equation}\label{eqn:Gen-Vuorinen-type-ineq-particular}
\frac{1}{2x} < \psi \left(x+\frac{1}{2} \right) -\psi(x) < \frac{1}{2x} + \frac{3}{2} - 2\ln2, \quad x\in(1,\infty) .
\end{equation}
\end{remark}

%-------------------Remark -------------------------
\begin{remark}
If $a=\frac{1}{2}$ and $k=2$ in Corollary \ref{cor:Gen-Vuorinen-type-ineq}, then we obtain
\begin{equation}\label{eqn:Gen-Vuorinen-type-ineq-particular}
\frac{1}{x^3} < \psi'' \left(x+\frac{1}{2} \right) -\psi''(x) < \frac{1}{x^3} + 15 - 12 \zeta(3), \quad x\in(1,\infty),
\end{equation}
where $\zeta(x)$ is the Riemann zeta function.
\end{remark}

%-------------------------- Theorem:--------------------------------------
\begin{theorem}\label{thm:Gen-Mortici-CM-Type-2}
For $a\in(0,1)$ and $k\in \{ 2n+1: n\in \mathbb{N}_0 \}$, let
\begin{equation}\label{eqn:Gen-Mortici-CM-Type-2}
h_{a,k}(x)=\psi^{(k)}(x+a) - \psi^{(k)}(x) - \frac{ak!}{x^{k+1}}.
\end{equation}
Then $-h_{a,k}(x)$ is strictly completely monotonic on $(0,\infty)$. 
\end{theorem}

\begin{proof}
Similarly, for $a\in(0,1)$ and $k\in \{ 2n+1: n\in \mathbb{N}_0 \}$, we have
\begin{align*}
- h_{a,k}^{(n)}(x) 
&=\frac{(-1)^{n}a(k+n)!}{x^{k+n+1}} + \psi^{(k+n)}(x)  - \psi^{(k+n)}(x+a) \\
&=  (-1)^na\int_{0}^{\infty}t^{k+n}e^{-xt}\,dt +  (-1)^{k+n+1}\int_{0}^{\infty}\frac{t^{k+n}e^{-xt}}{1-e^{-t}}\,dt  \\ & \quad - (-1)^{k+n+1}\int_{0}^{\infty}\frac{t^{k+n}e^{-(x+a)t}}{1-e^{-t}}\,dt .
\end{align*}
Then,
\begin{align*}
(-1)^n \left(- h_{a,k}\right)^{(n)}(x)
&= a\int_{0}^{\infty}t^{k+n}e^{-xt}\,dt + \int_{0}^{\infty}\frac{t^{k+n}e^{-xt}}{1-e^{-t}}\,dt - \int_{0}^{\infty}\frac{t^{k+n}e^{-xt}e^{-at}}{1-e^{-t}}\,dt   \\
&= \int_{0}^{\infty} \left[a + \frac{1-e^{-at}}{1-e^{-t}} \right]t^{k+n}e^{-xt}\,dt \\
&>0 ,
\end{align*}
which completes the proof.
\end{proof}

%-------------------Remark -------------------------
\begin{remark}
Since $-h_{a,k}(x)$ is strictly completely monotonic on $(0,\infty)$, it follows that $h_{a,k}(x)$ is strictly concave and strictly increasing on $(0,\infty)$. 
\end{remark}

%-------------------------- Corollary: k-Odd --------------------------------------
\begin{corollary}\label{cor:Gen-Vuorinen-type-ineq-Odd-k}
The inequality
\begin{equation}\label{eqn:Gen-Vuorinen-type-ineq-Odd-k}
\psi^{(k)}(a) - \psi^{(k)}(1) + k!\left(\frac{a}{x^{k+1}} - \frac{1}{a^{k+1}} - a\right) <
\psi^{(k)}(x+a) - \psi^{(k)}(x) < 
\frac{ak!}{x^{k+1}},
\end{equation}
holds for $a\in(0,1)$, $k\in \{ 2n+1: n\in \mathbb{N}_0 \}$ and $x\in(1,\infty)$.
\end{corollary}

\begin{proof}
Since $h_{a,k}(x)$  is increasing, then for $x\in(1,\infty)$, and by using \eqref{eqn:funct-eqn-polygamma}, we obtain
\begin{equation*}
\psi^{(k)}(a) - \psi^{(k)}(1) - \frac{k!}{a^{k+1}} - ak! =  h_{a,k}(1) <  h_{a,k}(x) < \lim_{x\rightarrow \infty}h_{a,k}(x)=0,
\end{equation*}
which yields \eqref{eqn:Gen-Vuorinen-type-ineq-Odd-k}.
\end{proof}

%-------------------Remark -------------------------
\begin{remark}
Particularly, if $a=\frac{1}{2}$ and $k=1$ in Corollary \ref{cor:Gen-Vuorinen-type-ineq-Odd-k}, then we obtain
\begin{equation}\label{eqn:Gen-Vuorinen-type-ineq-particular-Odd-k-1}
\frac{1}{2x^2}+\frac{\pi^2}{3} - \frac{9}{2} < \psi' \left(x+\frac{1}{2} \right) -\psi'(x) < \frac{1}{2x^2}, \quad x\in(1,\infty) .
\end{equation}
\end{remark}

%-------------------Remark -------------------------
\begin{remark}
Likewise, if $a=\frac{1}{2}$ and $k=3$ in Corollary \ref{cor:Gen-Vuorinen-type-ineq-Odd-k}, then we obtain
\begin{equation}\label{eqn:Gen-Vuorinen-type-ineq-particular-Odd-k-2}
\frac{3}{x^4} + \frac{14\pi^{4}}{15} - 99 < \psi''' \left(x+\frac{1}{2} \right) -\psi'''(x) < \frac{3}{x^4}, \quad x\in(1,\infty) .
\end{equation}
\end{remark}

%%%%%%%%%%%%%%%%%%%%%%%%%%%%%%%%%%%%%%%%%%%%%%%%%%%%
\section*{Conflicts of Interest}
%%%%%%%%%%%%%%%%%%%%%%%%%%%%%%%%%%%%%%%%%%%%%%%%%%%%

\noindent
The author declares that there are no conflicts of interest regarding the publication of this paper.

%----------------------------------------------------BIBLIOGRAGHY-------------------------------------------------------------
\bibliographystyle{plain}

%------------------------------------------------------------------------------------------------------------------------------------

\end{document}